\documentclass{article}

\usepackage[final]{graphicx}   
\usepackage{psfrag}   
\usepackage{amsmath,amsfonts,amssymb,amsxtra,subeqnarray}

\newcommand {\KL}{\ensuremath{\mathcal{K}\mathcal{L}}}
\newcommand {\K}{\ensuremath{\mathcal{K}}}
\newcommand {\Real}{\ensuremath{{\mathbb{R}}}}
\newcommand {\Natural}{\ensuremath{{\mathbb{N}}}}

\newcommand{\C}{\ensuremath{\mathcal C}}

\newcommand{\A}{\ensuremath{\mathcal A}}

\newcommand{\setE}{\ensuremath{\mathcal E}}

\newcommand{\G}{\ensuremath{\mathcal G}}
\newcommand{\N}{\ensuremath{\mathcal N}}

\newcommand{\ex}{\ensuremath{{\mathbf{x}}}}
\newcommand{\zi}{\ensuremath{{\mathbf{z}}}}

\newcommand{\vay}{\ensuremath{{\mathbf{y}}}}

\newtheorem{theorem}{Theorem}

\newtheorem{lemma}{Lemma}

\newtheorem{definition}{Definition}

\newenvironment{proof}{\noindent {\bf Proof.}}{\hfill \hspace*{1pt}\hfill$\blacksquare$}

\begin{document}

\title{Nonlinearly coupled harmonic oscillators:
high frequency oscillations yield synchronization} 
\author{S. Emre Tuna\\
{\em \small{Middle East Technical University, Ankara, Turkey}}}
\maketitle

\begin{abstract}
Synchronization of coupled harmonic oscillators is
investigated. Coupling considered here is pairwise, unidirectional,
and described by a nonlinear function (whose graph resides in the
first and third quadrants) of some projection of the relative distance
(between the states of the pair being coupled) vector.
Under the assumption that the interconnection topology defines a
connected graph, it is shown that the synchronization manifold is
semiglobally practically asymptotically stable in the frequency of
oscillations.
\end{abstract}

\section{Introduction}

Synchronization in coupled dynamical systems, due to the broad range
of applications, has been a common ground of investigation for
researchers from different disciplines. Most of the work in the area
studies the case where the interconnection between individual systems
is linear; see, for instance,
\cite{wu95,pecora98,pogromsky02,wu05,belykh06,liu08}.  Nonlinear
coupling is also of interest since certain phenomena cannot be
meaningfully modelled by linear coupling. A particular system
exemplifying nonlinear coupling that attracted much attention is
Kuramoto model and its like
\cite{bonilla98,pikovsky09}. Among more general results allowing
nonlinear coupling are \cite{arcak07,stan07} where passivity theory is
employed to obtain sufficient conditions for synchronization under
certain symmetry or balancedness assumptions on the graph describing
the interconnection topology.

In this paper we consider the basic equation of the theory of
oscillations 
\begin{equation*}
\ddot{q}=-\omega^{2}q
\end{equation*}
i.e., the harmonic oscillator. One physical example is a unit mass
attached to a spring. We interest ourselves with the following
question.  {\em If we take a number of identical mass-spring systems
and couple some pairs by unidirectional nonlinear dampers, will they
eventually oscillate synchronously?}  To be precise, do the solutions
of the following array of coupled harmonic oscillators
\begin{equation*}
\ddot{q}_{i}=-\omega^{2}q_{i}+\sum_{j\neq i}\gamma_{ij}(\dot{q}_{j}-\dot{q}_{i})\,,
\quad i=1,\,2,\,\ldots,\,p
\end{equation*}
(where $\gamma_{ij}(\cdot)$ is either identically zero or it is some
function whose graph lies entirely in the first and third quadrants)
synchronize? This question may be of direct importance for certain
simple mechanical systems or electrical circuits, but our main
interest in it is due to the possibility that it may serve as a lucky
starting point for understanding a more general scenario.

Regarding the above question, our finding in the paper is {\em
roughly} that synchronization occurs among mass-spring systems if the
springs are stiff enough and there is at least one system that
directly or indirectly effects all others. More formally, what we show
is that, for a given set of functions $\{\gamma_{ij}(\cdot)\}$
describing the coupling configuration, if the graph representing the
interconnection topology is connected, then the solutions can be made
to converge to an arbitrarily small neighborhood of the
synchronization manifold, starting from initial conditions arbitrarily
far from it by choosing large enough $\omega$. In technical terms,
what we establish is the semiglobal practical asymptotic stability
(in $\omega$) of the synchronization manifold. Intuition
and simulations tell us that global asymptotic synchronization should
occur regardless of what $\omega$ is. This however we have not been
able to prove (nor disprove).

We reach our final result in three steps. Note that the solution of an
uncoupled harmonic oscillator defines a rotating vector on the
plane. Thanks to linearity of the system the speed of this rotation
($\omega$) is independent of the initial conditions. As a first step
therefore we express the systems with respect to a rotating coordinate
system. This change of variables yields coupled systems whose
righthand sides are periodic in time (with period $2\pi/\omega$). Our
second step is to exploit this periodicity. We obtain the average
systems and realize that they belong to a well-studied class of
systems pertaining to {\em consensus} problems
\cite{lin07}. We then deduce that the solutions of average systems 
converge to a fixed point on the plane. As our final step we use the
result of \cite{teel99} to conclude that the global asymptotic
synchronization of average systems implies the semiglobal practical
asymptotic synchronization of coupled harmonic oscillators.
  
\section{Preliminaries}
Let $\Natural$ denote the set of nonnegative integers and $\Real_{\geq
0}$ the set of nonnegative real numbers. Let $|\!\cdot\!|$ denote
Euclidean norm. For $\ex=[x_{1}^{T}\ x_{2}^{T}\ \ldots\
x_{p}^{T}]^{T}$ with $x_{i}\in\Real^{n}$ we let
$\A:=\{\ex\in\Real^{np}:x_{i}=x_{j}\
\mbox{for all}\ i,\,j\}$ be {\em synchronization manifold}. A
function $\alpha:\Real_{\geq 0}\to
\Real_{\geq 0}$ is said to belong to class-{\K} $(\alpha\in\K)$ if it is 
continuous, zero at zero, and strictly increasing. A function
$\beta:\Real_{\geq 0}\times\Real_{\geq 0}\to\Real_{\geq 0}$ is said to
belong to class-$\KL$ if, for each $t\geq 0$, $\beta(\cdot,\,t)$ is
nondecreasing and $\lim_{s\to 0^{+}}\beta(s,\,t)=0$, and, for each
$s\geq 0$, $\beta(s,\,\cdot)$ is nonincreasing and
$\lim_{t\to\infty}\beta(s,\,t)=0$.  Given a closed set ${\mathcal
S}\subset\Real^{n}$ and a point $x\in\Real^{n}$, $|x|_{\mathcal S}$
denotes the (Euclidean) distance from $x$ to ${\mathcal S}$.

A ({\em directed}) {\em graph} is a pair $(\N,\,\setE)$ where $\N$ is
a nonempty finite set (of {\em nodes}) and $\setE$ is a finite
collection of ordered pairs ({\em edges}) $(n_{i},\,n_{j})$ with
$n_{i},\,n_{j}\in\N$. A {\em directed path} from $n_{1}$ to $n_{\ell}$
is a sequence of nodes $(n_{1},\,n_{2},\,\ldots,\,n_{\ell})$ such that
$(n_{i},\,n_{i+1})$ is an edge for
$i\in\{1,\,2,\,\ldots,\,\ell-1\}$. A graph is {\em connected} if it
has a node to which there exists a directed path from every other
node.\footnote{This is another way of saying that the graph contains a
spanning tree.} 

A set of functions $\{\gamma_{ij}:\Real\to\Real\}$, where
$i,\,j=1,\,2,\,\ldots,\,p$ with $i\neq j$, describes (is) an {\em
interconnection} if the following hold for all $i,\,j$ and all $s\in\Real$:
\vspace{-0.05in}
\begin{itemize}
\item[(i)] $\gamma_{ij}(0)=0$ and $s\gamma_{ij}(s)\geq 0$.
\vspace{-0.05in}
\item[(ii)] Either $\gamma_{ij}(s)\equiv 0$ or there exists $\alpha\in\K$ 
such that $|\gamma_{ij}(s)|\geq \alpha(|s|)$. 
\end{itemize}
\vspace{-0.05in}
To mean $\gamma_{ij}(s)\equiv 0$ we write $\gamma_{ij}=0$. Otherwise
we write $\gamma_{ij}\neq 0$. The graph of interconnection
$\{\gamma_{ij}\}$ is pair $(\N,\,\setE)$, where
$\N=\{n_{1},\,\ldots,\,n_{p}\}$ and $\setE$ is such that
$(n_{i},\,n_{j})\in\setE$ iff $\gamma_{ij}\neq 0$. An interconnection
is said to be {\em connected} when its graph is connected.

To give an example, consider a set of functions $\C:=\{\gamma_{ij}:
i,\,j=1,\ldots,\,4\}$.  Let
$\gamma_{13},\,\gamma_{23},\,\gamma_{24},\,\gamma_{32}$ be as in
Fig.~1 while the remaining functions be zero. Note
that each $\gamma_{ij}$ satisfies conditions (i) and (ii). Therefore
set $\C$ describes an interconnection. To determine whether $\C$ is
connected or not we examine its graph, see Fig.~2. Since
there exists a path to node $n_{4}$ from every other node, we deduce
that the graph (hence interconnection $\C$) is connected.

\begin{figure}[h]
\begin{center}
\includegraphics[scale=0.6]{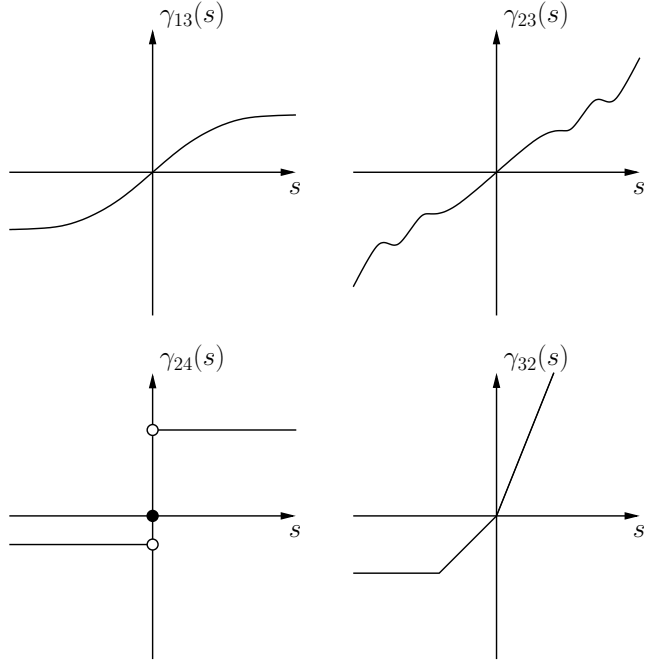}
\caption{Some examples of coupling functions.}
\end{center}
\end{figure}

\begin{figure}[h]
\begin{center}
\includegraphics[scale=0.6]{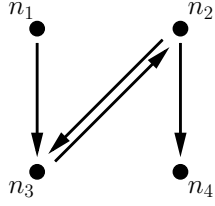}
\caption{Graph of interconnection $\C$.}
\end{center}
\end{figure}

\section{Problem statement}

We consider the following array of coupled harmonic oscillators
\begin{subeqnarray}\label{eqn:harmosc}
\dot{q}_{i}&=&\omega p_{i}\\
\dot{p}_{i}&=&-\omega q_{i}+
\sum_{j\neq i}\gamma_{ij}(p_{j}-p_{i})\,,
\qquad i=1,\,2,\,\ldots,\,p
\end{subeqnarray}
where $\omega>0$ and $\{\gamma_{ij}\}$ is a connected interconnection.
We assume throughout the paper that $\gamma_{ij}$ are locally
Lipschitz.  Let $\xi_{i}\in\Real^{2}$ denote the {\em state} of $i$th
oscillator, i.e., $\xi_{i}=[q_{i}\ p_{i}]^{T}$. When coupling
functions $\gamma_{ij}$ are linear, the oscillators are known to
(exponentially) synchronize for all $\omega$. That is, solutions
$\xi_{i}(\cdot)$ converge to a common (bounded) trajectory, see
\cite{ren08harmonic}. In this paper we investigate the behaviour of
oscillators under nonlinear coupling. In particular, we aim to
understand the effect of the frequency of oscillations $\omega$ on
synchronization for a given interconnection $\{\gamma_{ij}\}$.

\section{Change of coordinates}\label{sec:coc}
We define $S(\omega)\in\Real^{2\times 2}$ and $H\in\Real^{1\times
2}$ as
\begin{eqnarray*}
S(\omega):=
\left[\!\!
\begin{array}{rr}
0&\omega\\-\omega&0
\end{array}\!\!
\right]\,,\quad
H:=[0\ 1]\,.
\end{eqnarray*}
Then we rewrite \eqref{eqn:harmosc} as
\begin{eqnarray}\label{eqn:harmosc2}
\dot{\xi}_{i}=S(\omega)\xi_{i}+H^{T}\sum_{j\neq i}\gamma_{ij}(H(\xi_{j}-\xi_{i}))\,.
\end{eqnarray}
Let us recall the geometric meanings of the terms in
\eqref{eqn:harmosc2}. The first term $S(\omega)\xi_{i}$ defines a
{\em rotation} (with period $2\pi/\omega$) since $S(\omega)$ is a
skew-symmetric matrix. The second term
$H^{T}\sum\gamma_{ij}(H(\xi_{j}-\xi_{i}))$ induces a contraction in
the direction specified by vector $H^{T}$ (so to speak, along the {\em
vertical axis}) such that the {\em projections} of states $\xi_{i}$ on
the vertical axis tend to approach each other. The combined effect of
these two terms is relatively more difficult to visualize. One trick
to partially overcome this difficulty is to look at the system from
the point of view of the observer that sits on a rotating frame of
reference.

Under change of coordinates $x_{i}(t):=e^{-S(\omega)t}\xi_{i}(t)$ we
can by \eqref{eqn:harmosc2} write
\begin{eqnarray}\label{eqn:harmosc3}
\dot{x}_{i}&=&e^{S(\omega)^{T}t}H^{T}\sum_{j\neq i}\gamma_{ij}
(He^{S(\omega)t}(x_{j}-x_{i}))\nonumber\\
&=&\left[\!\!\begin{array}{r} -\sin\omega t\\ \cos\omega t
\end{array}\!\!\right]\sum_{j\neq i}\gamma_{ij}
([-\sin\omega t\ \cos\omega t](x_{j}-x_{i}))\,.
\end{eqnarray} 
Since the change of coordinates is realized via rotation matrix
$e^{-S(\omega)t}$, the relative distances are preserved, that is
$|x_{i}(t)-x_{j}(t)|=|\xi_{i}(t)-\xi_{j}(t)|$ for all $t$ and all
$i,\,j$. This means from the synchronization point of view that the
behaviour of array~\eqref{eqn:harmosc2} will be inherited by
array~\eqref{eqn:harmosc3}. However, exact analysis of
\eqref{eqn:harmosc3} seems still far from yielding. Therefore 
we attempt to understand this system via its approximation.

\section{Average systems}

Observe that the righthand side of
\eqref{eqn:harmosc3} is periodic in time. 
Time average functions $\bar\gamma_{ij}:\Real^{2}\to\Real^{2}$ are
given by
\begin{eqnarray*}
\bar{\gamma}_{ij}(x):=\frac{1}{2\pi}\int_{0}^{2\pi}
\left[\!\!\begin{array}{r} -\sin\varphi \\ \cos\varphi
\end{array}\!\!\right]\gamma_{ij}
([-\sin\varphi\ \cos\varphi]x)d\varphi\,.
\end{eqnarray*}
Then the {\em average} array dynamics read
\begin{eqnarray}\label{eqn:approx}
{\dot\eta}_{i}=\sum_{j\neq i} \bar\gamma_{ij}(\eta_{j}-\eta_{i})\,.
\end{eqnarray}
Theory of perturbations \cite[Ch.~4~\S~17]{arnold88} tells us that, starting
from close initial conditions, the solution of a system with a
periodic righthand side and the solution of the time-average
approximate system stay close for a long time provided that the period
is small enough. Therefore
\eqref{eqn:approx} should tell us a great deal about the 
behaviour of \eqref{eqn:harmosc3} when $\omega\gg 1$. 

Understanding \eqref{eqn:approx} requires understanding average
function $\bar\gamma_{ij}$. The following lemma is helpful from
that respect.

\begin{lemma}\label{lem:key}
We have $\displaystyle \bar\gamma_{ij}(x)=\rho_{ij}(|x|)\frac{x}{|x|}$ where 
$\rho_{ij}:\Real_{\geq 0}\to\Real_{\geq 0}$ is 
\begin{eqnarray*}\label{eqn:key}
\rho_{ij}(r):=\frac{1}{2\pi}\int_{0}^{2\pi}
\gamma_{ij}(r\sin\varphi)\sin\varphi\,d\varphi\,.
\end{eqnarray*}
\end{lemma}
\begin{proof}
Given $x\in\Real^{2}$, let $r=|x|$ and $\theta\in[0,\,2\pi)$ be such
that $r[-\cos\theta\ \sin\theta]^{T}=x$. Then, by using standard
trigonometric identities,
\begin{eqnarray}\label{eqn:from}
\bar\gamma_{ij}(x)
&=&\frac{1}{2\pi}\int_{0}^{2\pi}
\left[\!\!\begin{array}{r} -\sin\varphi \\ \cos\varphi
\end{array}\!\!\right]\gamma_{ij}
(r(\sin\varphi\cos\theta+\cos\varphi\sin\theta))d\varphi\nonumber\\
&=&\frac{1}{2\pi}\int_{0}^{2\pi}
\left[\!\!\begin{array}{r} -\sin\varphi \\ \cos\varphi
\end{array}\!\!\right]\gamma_{ij}
(r\sin(\varphi+\theta))d\varphi\nonumber\\
&=&\frac{1}{2\pi}\int_{0}^{2\pi}
\left[\!\!\begin{array}{r} -\sin(\varphi-\theta) \\ \cos(\varphi-\theta)
\end{array}\!\!\right]\gamma_{ij}
(r\sin\varphi)d\varphi\nonumber\\
&=&\frac{1}{2\pi}\int_{0}^{2\pi}
\left[\!\!\begin{array}{r} -\sin\varphi\cos\theta+\cos\varphi\sin\theta \\ 
\cos\varphi\cos\theta+\sin\varphi\sin\theta
\end{array}\!\!\right]\gamma_{ij}
(r\sin\varphi)d\varphi\nonumber\\
&=&
\left(
\frac{1}{2\pi}\int_{0}^{2\pi}
\gamma_{ij}
(r\sin\varphi)\sin\varphi\,d\varphi\right)
\left[\!\!\begin{array}{r} -\cos\theta \\ 
\sin\theta
\end{array}\!\!\right]\nonumber\\
&&\qquad +\left(
\frac{1}{2\pi}\int_{0}^{2\pi}
\gamma_{ij}
(r\sin\varphi)\cos\varphi\,d\varphi\right)
\left[\!\!\begin{array}{c} \sin\theta \\ 
\cos\theta
\end{array}\!\!\right]
\end{eqnarray}
We focus on the second term in \eqref{eqn:from}. Observe 
that
\begin{eqnarray*}
\int_{0}^{\pi}\gamma_{ij}(r\sin\varphi)\cos\varphi\,d\varphi
&=&\int_{-\pi/2}^{\pi/2}
\gamma_{ij}\left(r\sin\left(\varphi+\frac{\pi}{2}\right)\right)
\cos\left(\varphi+\frac{\pi}{2}\right)d\varphi\\
&=&0
\end{eqnarray*}
since the integrand is an odd function on the interval of integration.
Likewise,
\begin{eqnarray*}
\int_{\pi}^{2\pi}\gamma_{ij}(r\sin\varphi)\cos\varphi\,d\varphi
&=&\int_{-\pi/2}^{\pi/2}
\gamma_{ij}\left(r\sin\left(\varphi+\frac{3\pi}{2}\right)\right)
\cos\left(\varphi+\frac{3\pi}{2}\right)d\varphi\\
&=&0\,.
\end{eqnarray*}
Therefore
\begin{eqnarray}\label{eqn:from2}
\lefteqn{\int_{0}^{2\pi}\gamma_{ij}(r\sin\varphi)\cos\varphi\,d\varphi}\nonumber\\
&&=\int_{0}^{\pi}\gamma_{ij}(r\sin\varphi)\cos\varphi\,d\varphi
+\int_{\pi}^{2\pi}\gamma_{ij}(r\sin\varphi)\cos\varphi\,d\varphi\nonumber\\
&&=0\,.
\end{eqnarray}
Combining \eqref{eqn:from} and \eqref{eqn:from2} we obtain
\begin{eqnarray*}
\bar\gamma_{ij}(x)
&=&
\left(
\frac{1}{2\pi}\int_{0}^{2\pi}
\gamma_{ij}
(r\sin\varphi)\sin\varphi\,d\varphi\right)
\left[\!\!\begin{array}{r} -\cos\theta \\ 
\sin\theta
\end{array}\!\!\right]\\
&=&\rho_{ij}(|x|)\frac{x}{|x|}\,.
\end{eqnarray*}
Hence the result.
\end{proof}
\\

Lemma~\ref{lem:key} tells us that, given any vector $x$ on the plane,
vector $\bar\gamma_{ij}(x)$ (if it has nonzero magnitude) is in the
same direction as $x$. In the light of this if we now look at
systems~\eqref{eqn:approx} we can roughly visualize the evolution of
the trajectories. Take the $i$th system. Let
$j_{1},\,\ldots,\,j_{\ell}$ be all the indices such that
$j\in\{j_{1},\,\ldots,\,j_{\ell}\}$ implies $\gamma_{ij}\neq 0$.
These indices are sometimes called the indices of {\em neighbors} of
system $i$. Then we can write $\dot\eta_{i}=v_{1}+\ldots+v_{\ell}$
where each $v_{k}$ is a vector pointing from $\eta_{i}$ to the state
of $k$th neighbor system. Therefore the net velocity vector $\sum
v_{k}$ points to some ``weighted mean'' of the neighbor systems'
states. Synchronization of such systems, where the velocity vector of
a system always points to some weighted mean of the positions of its
neighbors, have been studied under the names {\em consensus} and {\em
state agreement}; see, for instance,
\cite{moreau05,angeli06,lin07}. The finding of those works is roughly
that if the interconnection is connected, then the solutions of
systems converge to a common fixed point in space.  We now give the
formal application of this result to our case.

Let us stack the individual states $\eta_{i}$ to form
$\eta:=[\eta_{1}^{T}\
\eta_{2}^{T}\ \ldots\ \eta_{p}^{T}]^{T}$.  Define
\begin{eqnarray*}
\gamma_{\rm av}(\eta):=
\left[\!\!
\begin{array}{c}
\sum\bar\gamma_{1j}(\eta_{j}-\eta_{1})\\
\vdots\\
\sum\bar\gamma_{pj}(\eta_{j}-\eta_{p})
\end{array}\!\!
\right]
\end{eqnarray*}
We then reexpress \eqref{eqn:approx} as
\begin{eqnarray}\label{eqn:eta}
\dot\eta=\gamma_{\rm av}(\eta)\,.
\end{eqnarray}
We now have the following result.
\begin{theorem}\label{thm:KL}
Consider system~\eqref{eqn:eta}. Synchronization manifold $\A$ is
globally asymptotically stable, i.e., there exists $\beta\in\KL$ such
that $|\eta(t)|_{\A}\leq\beta(|\eta(0)|_{\A},\,t)$ for all $t\geq 0$.
\end{theorem}

\begin{proof}
By Lemma~\ref{lem:key} we can write
\begin{eqnarray*}
\dot\eta_{i}=\sum_{j\neq i} \rho_{ij}
(|\eta_{j}-\eta_{i}|)\frac{\eta_{j}-\eta_{i}}{|\eta_{j}-\eta_{i}|}=:f_{i}(\eta)
\end{eqnarray*}
for $i=1,\,2,\,\ldots,\,p$. Let $\G$ denote the graph of interconnection
$\{\gamma_{ij}\}$. Note that $\G$ is connected by assumption.  We make
the following simple observations. Function $\rho_{ij}$ is
continuous and zero at zero. If there is no edge of $\G$ from node $i$
to node $j$ then $\rho_{ij}(r)\equiv 0$. If there is an edge from node
$i$ to node $j$ then $\sigma_{ij}(r)>0$ for $r>0$. 

Therefore $f_{i}$ is continuous; and vector $f_{i}(\eta)$
always points to the (relative) interior of the convex hull of the set
$\{\eta_{i}\}\cup\{\eta_{j}:\mbox{there is an edge of $\G$ from node
$i$}$ $\mbox{to node $j$}\}$. These two conditions together with
connectedness of $\G$ yield by \cite[Corollary~3.9]{lin07} that
system~\eqref{eqn:eta} has the {\em globally asymptotic state
agreement property}, see \cite[Definition~3.4]{lin07}. Another
property of the system is {\em invariance with respect to
translations}. That is, $\gamma_{\rm av}(\eta+\xi)=\gamma_{\rm
av}(\eta)$ for $\xi\in\A$. These properties let us write the
following.
\begin{itemize}
\item[(a)] There exists a class-$\K$ function $\alpha$ such that 
$|\eta(t)|_{\A}\leq\alpha(|\eta(0)|_{\A})$ for all $t\geq 0$.
\item[(b)] For each $r>0$ and $\varepsilon>0$, there exists $T>0$ such that
$|\eta(0)|_{\A}\leq r$ implies $|\eta(t)|_{\A}\leq\varepsilon$ for all
$t\geq T$.
\end{itemize}
Finally, (a) and (b) give us the result by
\cite[Proposition~1]{teel00}.
\end{proof}
\\

Let us now go back to our discussion in the beginning of
Section~\ref{sec:coc}.  There we talked about two actions that shape
the dynamics of systems~\eqref{eqn:harmosc2}, namely, rotation and
vertical contraction. The combined effect of those actions on
synchronization of the systems was not initially apparent. However, by
applying first a change of coordinates \eqref{eqn:harmosc3} and then
averaging
\eqref{eqn:approx} we see that it is likely that two actions will result 
in synchronization, at least when the rotation is rapid
enough. Vaguely speaking, rotation rescues contraction from being
confined only to vertical direction and sort of {\em smears} it
uniformly to all directions, which should bring synchronization. In
the next section we formalize our observation.

\section{Semiglobal practical asymptotic synchronization}

Consider systems~\eqref{eqn:harmosc3}. Stack states $x_{i}$ to
form $\ex:=[x_{1}^{T}\ x_{2}^{T}\ \ldots\ x_{p}^{T}]^{T}$. Define
\begin{eqnarray*}
\gamma(\ex,\,\omega t):=
\left[\!\!
\begin{array}{c}
\left[\!\!\begin{array}{r} -\sin\omega t\\ \cos\omega t
\end{array}\!\!\right]\sum\gamma_{1j}
([-\sin\omega t\ \cos\omega t](x_{j}-x_{1}))\\
\vdots\\
\left[\!\!\begin{array}{r} -\sin\omega t\\ \cos\omega t
\end{array}\!\!\right]\sum\gamma_{pj}
([-\sin\omega t\ \cos\omega t](x_{j}-x_{p}))
\end{array}
\!\!\right]
\end{eqnarray*}
Now reexpress \eqref{eqn:harmosc3} as
\begin{eqnarray}\label{eqn:gamma}
\dot\ex=\gamma(\ex,\,\omega t)\,.
\end{eqnarray}
The following definition is borrowed with slight modification from
\cite{teel99}.

\begin{definition}
Consider system $\dot\ex=f(\ex,\,\omega t)$. Closed set ${\mathcal S}$
is said to be {\em semiglobally practically asymptotically stable} if
for each pair $(\Delta,\,\delta)$ of positive numbers, there exists
$\omega^{*}>0$ such that for each $\omega\geq\omega^{*}$ the following
hold.
\begin{itemize}
\item[(a)] For each $r>\delta$ there exists $\varepsilon>0$ such that
\begin{equation*}
|\ex(0)|_{\mathcal S}\leq \varepsilon \implies |\ex(t)|_{\mathcal S}\leq r \quad 
\forall t\geq 0\,.
\end{equation*}
\item[(b)] For each $\varepsilon<\Delta$ there exists $r>0$ such that
\begin{equation*}
|\ex(0)|_{\mathcal S}\leq \varepsilon \implies  |\ex(t)|_{\mathcal S}\leq r \quad
\forall t\geq 0\,.
\end{equation*}
\item[(c)] For each $r<\Delta$ and $\varepsilon>\delta$ 
there exists $T>0$ such that 
\begin{equation*}
|\ex(0)|_{\mathcal S}\leq r \implies 
|\ex(t)|_{\mathcal S}\leq\varepsilon \quad
\forall t\geq T\,.
\end{equation*}
\end{itemize}
\end{definition}
Below we establish the semiglobal practical asymptotic stability of
synchronization manifold. To do that we use \cite[Thm.~2]{teel99},
which says that the origin of system $\dot\ex=f(\ex,\,\omega t)$ (where $f$
is periodic in time) is semiglobally practically asymptotically stable
if the origin of $\dot\ex=f_{\rm av}(\ex)$ (where $f_{\rm av}$
is the time average of $f$) is globally asymptotically stable.
\begin{theorem}\label{thm:main}
Consider system~\eqref{eqn:gamma}. Synchronization manifold $\A$ is
semiglobally practically asymptotically stable.
\end{theorem}
\begin{proof}
Consider systems~\eqref{eqn:harmosc3}. Observe that the righthand side
depends only on the relative distances $x_{j}-x_{i}$. Let us define
\begin{eqnarray*}
\vay:=\left[\!\!
\begin{array}{c}
x_{2}-x_{1}\\
x_{3}-x_{1}\\
\vdots\\
x_{p}-x_{1}
\end{array}
\!\!\right]
\end{eqnarray*}
Note that $\dot\vay=f(\vay,\,\omega t)$ for some $f:\Real^{2p-2}
\times\Real_{\geq 0}\to \Real^{2p-2}$. Since functions 
$\gamma_{ij}$ are assumed to be locally Lipschitz,
$f$ is locally Lipschitz in $\vay$ uniformly in $t$.
Also, $f$ is periodic in time by
\eqref{eqn:harmosc3}. Now consider systems~\eqref{eqn:approx}. Again
the righthand side depends only on the relative distances
$\eta_{j}-\eta_{i}$. Define
\begin{eqnarray*}
\zi:=\left[\!\!
\begin{array}{c}
\eta_{2}-\eta_{1}\\
\eta_{3}-\eta_{1}\\
\vdots\\
\eta_{p}-\eta_{1}
\end{array}
\!\!\right]
\end{eqnarray*}
Then $\dot\zi=f_{\rm av}(\zi)$ where $f_{\rm av}$ is the time
average of $f$ and locally Lipschitz both due to that
$\bar\gamma_{ij}$ is the time average of $\gamma_{ij}$.

Theorem~\ref{thm:KL} implies that the origin of $\dot\zi=f_{\rm
av}(\zi)$ is globally asymptotically stable. Then
\cite[Thm.~2]{teel99} tells us that the origin of
$\dot\vay=f(\vay,\,\omega t)$ is semiglobally practically
asymptotically stable. All there is left to complete the proof is to
realize that semiglobal practical asymptotic stability of the origin
of $\dot\vay=f(\vay,\,\omega t)$ is equivalent to semiglobal practical
asymptotic stability of synchronization manifold $\A$ of
system~\eqref{eqn:gamma}.
\end{proof}
\\

Theorem~\ref{thm:main} can be recast into the following form.

\begin{theorem}\label{thm:direct}
Consider coupled harmonic oscillators~\eqref{eqn:harmosc}.  For each
pair $(\Delta,\,\delta)$ of positive numbers, there exists
$\omega^{*}>0$ such that for each $\omega\geq\omega^{*}$ the following
hold.
\begin{itemize}
\item[(a)] For each $r>\delta$ there exists $\varepsilon>0$ such that
\begin{equation*}
\max_{i,\,j}|\xi_{i}(0)-\xi_{j}(0)|\leq \varepsilon \implies 
\max_{i,\,j}|\xi_{i}(t)-\xi_{j}(t)|\leq r \quad
\forall t\geq 0\,.
\end{equation*}
\item[(b)] For each $\varepsilon<\Delta$ there exists 
$r>0$ such that
\begin{equation*}
\max_{i,\,j}|\xi_{i}(0)-\xi_{j}(0)|\leq \varepsilon
\implies \max_{i,\,j}|\xi_{i}(t)-\xi_{j}(t)|\leq r\quad 
\forall t\geq 0\,.
\end{equation*}
\item[(c)] For each $r<\Delta$ and $\varepsilon>\delta$ 
there exists $T>0$
such that 
\begin{equation*}
\max_{i,\,j}|\xi_{i}(0)-\xi_{j}(0)|\leq r \implies 
\max_{i,\,j}|\xi_{i}(t)-\xi_{j}(t)|\leq\varepsilon \quad 
\forall t\geq T\,.
\end{equation*}
\end{itemize}
\end{theorem} 

\section{Conclusion}

For nonlinearly coupled harmonic oscillators we have shown that
synchronization manifold is semiglobally practically asymptotically
stable in the frequency of oscillations. Our assumption on each coupling
function is that it is locally Lipschitz and, if nonzero, its graph
lies in the first and third quadrants and does not get arbitrarily
close to the horizontal axis when far from the origin. Our assumption
on the interconnection graph is the minimum; that is, it is connected.

One last remark we want to make is the following. If we look at
\eqref{eqn:harmosc2} we realize that $H=[0\ 1]$ is not necessary for
the rest of the analysis. In fact any nonzero $H\in\Real^{1\times 2}$
is no worse than $[0\ 1]$. For instance, for $H=[1\ 0]$ coupled
harmonic oscillators would be represented by
\begin{eqnarray*}
\dot{q}_{i}&=&\omega p_{i}+\sum_{j\neq i}\gamma_{ij}(q_{j}-q_{i})\\
\dot{p}_{i}&=&-\omega q_{i}
\end{eqnarray*}
for which Theorem~\ref{thm:direct} is valid.
 
\bibliographystyle{plain} 
\bibliography{references}

\end{document}